\documentclass[12pt]{article}
\usepackage{amsmath,amssymb,amsthm}
\numberwithin{equation}{section}
\usepackage{units}
\usepackage{color}
\usepackage[T1]{fontenc}
\usepackage[utf8]{inputenc}
\usepackage{authblk}
\usepackage{bm}
\usepackage{graphicx}

\usepackage{datetime}

\usepackage[colorlinks=true,
linkcolor=webgreen,
filecolor=webbrown,
citecolor=webgreen]{hyperref}

\definecolor{webgreen}{rgb}{0,.5,0}
\definecolor{webbrown}{rgb}{.6,0,0}

\usepackage[toc,page]{appendix}

\newcommand{\Z}{{\mathbb Z}}

\newtheorem{thm}{Theorem}
\newtheorem{theorem}[thm]{Theorem}
\newtheorem{lemma}{Lemma}

\newtheorem{corollary}[thm]{Corollary}

\title{Generalizations of Fibonacci-Lucas inverse tangent summation identities of Hoggatt and Ruggels}
\author[]{Kunle Adegoke \\\href{mailto:kunle.adegoke@yandex.com}{\tt kunle.adegoke@yandex.com}}

\affil{Department of Physics and Engineering Physics, \mbox{Obafemi Awolowo University}, 220005 Ile-Ife, Nigeria}

\begin{document}
\date{}

\maketitle

\begin{abstract}\noindent
We derive generalizations of a couple of inverse tangent summation identities involving Fibonacci and Lucas numbers. As byproducts we establish many new inverse tangent identities involving the Fibonacci and Lucas numbers.
\end{abstract}
\section{Introduction}
The Fibonacci numbers, $F_n$, and the Lucas numbers, $L_n$, $n\in\Z$, are defined by:
\begin{equation}
F_0  = 0,\;F_1  = 1,\;F_n  = F_{n - 1}  + F_{n - 2}\; (n\ge 2),\quad F_{ - n}  = ( - 1)^{n-1} F_n
\end{equation}
and
\begin{equation}
L_0  = 2,\;L_1  = 1,\;L_n  = L_{n - 1}  + L_{n - 2}\; (n\ge 2) ,\quad L_{ - n}  = ( - 1)^n L_n\,.
\end{equation}
Many inverse tangent summation identities involving the Fibonacci and Lucas numbers are known in the Literature (see Adegoke \cite{adegoke15}, Frontczak \cite{frontczak15}, Hoggatt and Ruggles \cite{ruggles63, ruggles64}, Melham \cite{melham99, melham95}).

\medskip

Using mathematical induction, Hoggatt and Ruggles \cite[Theorem 5]{ruggles63} proved
\begin{equation}\label{eq.qgvdh1o}
\sum\limits_{n = 1}^t {( - 1)^{n + 1} \tan ^{ - 1} \frac{1}{{F_{2n} }}}  = \tan ^{ - 1} \frac{{F_t }}{{F_{t + 1} }}\,,
\end{equation}
for non-negative integers $t$. 

\medskip

In the limit as $t\to\infty$, identity \eqref{eq.qgvdh1o} gives
\begin{equation}\label{eq.axhbhhg}
\sum\limits_{n = 1}^\infty {( - 1)^{n + 1} \tan ^{ - 1} \frac{1}{{F_{2n} }}}  = \tan ^{ - 1} \frac{1}{\varphi}\,,
\end{equation}
where $\varphi=(1+\sqrt 5)/2$ is the golden ratio.

\medskip

Hoggatt and Ruggles \cite{ruggles64} also derived
\begin{equation}\label{eq.ewqu221}
2\sum\limits_{n = 1}^t {\tan ^{ - 1} \frac{1}{{L_{2n} }}}  = \sum\limits_{n = 1}^t {\tan ^{ - 1} \frac{1}{{F_{2n + 1} }}}  - \tan ^{ - 1} \frac{1}{{L_{2t + 2} }} + \tan ^{ - 1} \frac{1}{3}\,,
\end{equation}
which gives
\begin{equation}
2\sum\limits_{n = 1}^\infty {\tan ^{ - 1} \frac{1}{{L_{2n} }}}  = \tan ^{ - 1} 2\,,
\end{equation}
since (Lehmer's formula, Hoggatt and Ruggles \cite[Theorem 5]{ruggles64})
\begin{equation}
\sum\limits_{n = 1}^\infty  {\tan ^{ - 1} \frac{1}{{F_{2n + 1} }}}  = \frac{\pi }{4}
\end{equation}
and
\begin{equation}
\frac{\pi }{4} + \tan ^{ - 1} \frac{1}{3} = \tan ^{ - 1} 2\,.
\end{equation}
In this paper we offer, for $m$ and $t$ non-negative integers, the following generalizations of \eqref{eq.qgvdh1o} and \eqref{eq.ewqu221}:
\[
\begin{split}
2\sum\limits_{n = 1}^t {( - 1)^{n - 1} \tan ^{ - 1} \frac{{F_m }}{{F_{2n + m - 1} }}}  &= \sum\limits_{n = 1}^m {( - 1)^{n - 1} \tan ^{ - 1} \frac{2}{{L_{2n - 1} }}}\\
&\quad  + ( - 1)^{t - 1} \sum\limits_{n = 1}^m {( - 1)^{n - 1} \tan ^{ - 1} \frac{{L_m }}{{L_{2n + 2t + m - 1} }}}\\
&\qquad- \sum\limits_{n = t + 1}^{t + m} {( - 1)^{n - 1} \tan ^{ - 1} \frac{{F_m }}{{F_{2n + m - 1} }}},\quad\text{$m=0,1,2,\ldots$}\,,
\end{split}
\]

\[
\begin{split}
2\sum\limits_{n = 1}^t {\tan ^{ - 1} \frac{{L_m }}{{L_{2n + m - 1} }}}  &= \sum\limits_{n = 1}^m {\tan ^{ - 1} \frac{2}{{L_{2n - 1} }}}  - \sum\limits_{n = 1}^m {\tan ^{ - 1} \frac{{F_m }}{{F_{2n + 2t + m - 1} }}}\\
&\qquad- \sum\limits_{n = t + 1}^{t + m} {\tan ^{ - 1} \frac{{L_m }}{{L_{2n + m - 1} }}},\quad\text{$m=1,3,5,\ldots$}\,,
\end{split}
\]
with the limiting values:
\[
2\sum\limits_{n = 1}^\infty {( - 1)^{n - 1} \tan ^{ - 1} \frac{{F_m }}{{F_{2n + m - 1} }}}  = \sum\limits_{n = 1}^m {( - 1)^{n - 1} \tan ^{ - 1} \frac{2}{{L_{2n - 1} }}},\quad\text{$m=0,1,2,\ldots$}\,,
\]

\[
2\sum\limits_{n = 1}^\infty {\tan ^{ - 1} \frac{{L_m }}{{L_{2n + m - 1} }}}  = \sum\limits_{n = 1}^m {\tan ^{ - 1} \frac{2}{{L_{2n - 1} }}}, \quad\text{$m=1,3,5,\ldots,$}\,.
\]
If $m$ is even, we show that
\[
\begin{split}
2\sum\limits_{n = 1}^t {\tan ^{ - 1} \frac{{F_m }}{{F_{2n + m - 1} }}}  &= \sum\limits_{n = 1}^m {\tan ^{ - 1} \frac{2}{{L_{2n - 1} }}}  - \sum\limits_{n = 1}^m {\tan ^{ - 1} \frac{{L_m }}{{L_{2n + 2t + m - 1} }}}\\
&\qquad- \sum\limits_{n = t + 1}^{t + m} {\tan ^{ - 1} \frac{{F_m }}{{F_{2n + m - 1} }}}\,,
\end{split}
\]
with the limiting value:
\[
2\sum\limits_{n = 1}^\infty {\tan ^{ - 1} \frac{{F_m }}{{F_{2n + m - 1} }}}  = \sum\limits_{n = 1}^m { \tan ^{ - 1} \frac{2}{{L_{2n - 1} }}},\text{ $m=0,2,4,\ldots$}\,.
\]
We require the following identities which should be familiar or are easily derivable from the multiplication and addition formulas of Fibonacci and Lucas numbers:
\begin{equation}\label{eq.jms631k}
F_{2m}=F_mL_m\,,
\end{equation}
\begin{equation}\label{eq.vdi7cbj}
F_nL_m+L_nF_m=2F_{m+n}\,,
\end{equation}
\begin{equation}\label{eq.t28c16c}
F_{n + 2m}  - F_n  = L_m F_{n + m},\quad\text{$m$ odd}\,, 
\end{equation}
\begin{equation}
F_{n + 2m}  - F_n  = F_m L_{n + m},\quad\text{$m$ even}\,, 
\end{equation}
\begin{equation}
F_{n + 2m}  + F_n  = F_m L_{n + m},\quad\text{$m$ odd}\,, 
\end{equation}
\begin{equation}\label{eq.gft31k}
F_{n + 2m}  + F_n  = L_m F_{n + m},\quad\text{$m$ even}\,, 
\end{equation}
\begin{equation}
L_{n + 2m}  - L_n  = L_m L_{n + m},\quad\text{$m$ odd}\,, 
\end{equation}
\begin{equation}
L_{n + 2m}  - L_n  = 5F_m F_{n + m},\quad\text{$m$ even}\,, 
\end{equation}
\begin{equation}
L_{n + 2m}  + L_n  = 5F_m F_{n + m},\quad\text{$m$ odd}\,, 
\end{equation}
\begin{equation}
L_{n + 2m}  + L_n  = L_m L_{n + m},\quad\text{$m$ even}\,, 
\end{equation}
\begin{equation}\label{eq.qoo6doj}
F_n F_{n + 2m}  = F_{n + m}^2  + F_m^2,\quad\text{$n$ odd}\,, 
\end{equation}
\begin{equation}
F_n F_{n + 2m}  = F_{n + m}^2  - F_m^2,\quad\text{$n$ even}\,, 
\end{equation}
\begin{equation}
L_n L_{n + 2m}  = 5F_{n + m}^2  - L_m^2,\quad\text{$n$ odd}\,,
\end{equation}
\begin{equation}\label{eq.c5jimqp}
L_n L_{n + 2m}  = 5F_{n + m}^2  + L_m^2,\quad\text{$n$ even}\,. 
\end{equation}
We also require the following inverse tangent identities:
\begin{equation}\label{eq.pj13if2}
\tan ^{ - 1} \frac{{x + y}}{{1 - xy}} = \tan ^{ - 1} x + \tan ^{ - 1} y\,,\quad\text{if $xy<1$}\,,
\end{equation}
\begin{equation}\label{eq.x7ooe1a}
\tan ^{ - 1} \frac{{x - y}}{{1 + xy}} = \tan ^{ - 1} x - \tan ^{ - 1} y\,,\quad\text{if $xy>-1$}\,.
\end{equation}
\section{Inverse tangent Fibonacci-Lucas identities}
As a consequence of identities \eqref{eq.jms631k} to \eqref{eq.c5jimqp} and the inverse tangent identities \eqref{eq.pj13if2} and \eqref{eq.x7ooe1a}, we have:
\begin{lemma}\label{lemma.id}
The following identities hold for non-negative integers $m$ and $n$:
\begin{equation}\label{eq.gdtp5sa}
\tan ^{ - 1} \frac{{F_{2m} }}{{F_{2n + 2m - 1} }} = \tan ^{ - 1} \frac{{L_m }}{{L_{2n + m - 1} }} - \tan ^{ - 1} \frac{{L_m }}{{L_{2n + 3m - 1} }},\quad\text{$m$ even}\,,
\end{equation}

\begin{equation}\label{eq.ryx4qng}
\tan ^{ - 1} \frac{{F_{2m} }}{{F_{2n + 2m - 1} }} = \tan ^{ - 1} \frac{{L_m }}{{L_{2n + m - 1} }} + \tan ^{ - 1} \frac{{L_m }}{{L_{2n + 3m - 1} }},\quad\text{$m$ odd}\,,
\end{equation}

\begin{equation}\label{eq.jf9ybp9}
\tan ^{ - 1} \frac{{F_{2m} }}{{F_{2n + 2m - 1} }} = \tan ^{ - 1} \frac{{F_m }}{{F_{2n + m - 1} }} - \tan ^{ - 1} \frac{{F_m }}{{F_{2n + 3m - 1} }},\quad\text{$m$ odd}\,,
\end{equation}

\begin{equation}\label{eq.imhso0b}
\tan ^{ - 1} \frac{{F_{2m} }}{{F_{2n + 2m - 1} }} = \tan ^{ - 1} \frac{{F_m }}{{F_{2n + m - 1} }} + \tan ^{ - 1} \frac{{F_m }}{{F_{2n + 3m - 1} }}\,,\quad\text{$m$ even}\,,
\end{equation}

\begin{equation}\label{eq.ec853w8}
\tan ^{ - 1} \frac{{L_m^2 L_{2n + 2m} }}{{5F_{2n + 2m}^2 }} = \tan ^{ - 1} \frac{{L_m }}{{L_{2n + m} }} - \tan ^{ - 1} \frac{{L_m }}{{L_{2n + 3m} }}\,,\quad\text{$m$ odd}\,,
\end{equation}

\begin{equation}\label{eq.ec853w8}
\tan ^{ - 1} \frac{{L_m^2 L_{2n + 2m} }}{{5F_{2n + 2m}^2 }} = \tan ^{ - 1} \frac{{L_m }}{{L_{2n + m} }} + \tan ^{ - 1} \frac{{L_m }}{{L_{2n + 3m} }}\,,\quad\text{$m$ even}\,,
\end{equation}

\begin{equation}\label{eq.f94zdr0}
\tan ^{ - 1} \frac{{F_m^2 L_{2n + 2m} }}{{F_{2n + 2m}^2 }} = \tan ^{ - 1} \frac{{F_m }}{{F_{2n + m} }} - \tan ^{ - 1} \frac{{F_m }}{{F_{2n + 3m} }}\,,\quad\text{$m$ even}\,,
\end{equation}

\begin{equation}\label{eq.i97s6g6}
\tan ^{ - 1} \frac{{F_m^2 L_{2n + 2m} }}{{F_{2n + 2m}^2 }} = \tan ^{ - 1} \frac{{F_m }}{{F_{2n + m} }} + \tan ^{ - 1} \frac{{F_m }}{{F_{2n + 3m} }}\,,\quad\text{$m$ odd}\,.
\end{equation}
\end{lemma}
We also have
\begin{equation}\label{eq.gxld0li}
\tan ^{ - 1} \frac{2}{{L_{2n - 1} }} = \tan ^{ - 1} \frac{{L_m }}{{L_{2n + m - 1} }} + \tan ^{ - 1} \frac{{F_m }}{{F_{2n + m - 1} }}\,.
\end{equation}
In order to obtain the summation identities associated with the identities in Lemma \ref{lemma.id}, we require the following telescoping summation identities which hold for any sequence~$(X_i)$ of real numbers (see Adegoke \cite{adegoke}):
\begin{equation}\label{eq.s1qn9qb}
\sum\limits_{n = 1}^k {\left( {X_n  - X_{n + m} } \right)}  = \sum\limits_{n = 1}^m {\left( {X_n  - X_{n + k} } \right)}\,,
\end{equation}
\begin{equation}\label{eq.n5lm399}
\sum\limits_{n = 1}^k {( - 1)^{n - 1} \left( {X_n  - X_{n + m} } \right)}  = \sum\limits_{n = 1}^m {( - 1)^{n - 1} X_n }  + ( - 1)^{k - 1} \sum\limits_{n = 1}^m {( - 1)^{n - 1} X_{n + k} },\quad\text{if $m$ is even}
\end{equation}
and
\begin{equation}\label{eq.e3jcedw}
\sum\limits_{n = 1}^k {( - 1)^{n - 1} \left( {X_n  + X_{n + m} } \right)}  = \sum\limits_{n = 1}^m {( - 1)^{n - 1} X_n }  + ( - 1)^{k - 1} \sum\limits_{n = 1}^m {( - 1)^{n - 1} X_{n + k} },\quad\text{if $m$ is odd}\,.
\end{equation}
Infinite summations are evaluated using
\begin{equation}
\sum\limits_{n = 1}^\infty {\left( {X_n  - X_{n + m} } \right)}  = \sum\limits_{n = 1}^m {X_n }  - \sum\limits_{n = 1}^m {\mathop {\lim }\limits_{k \to \infty } X_{n + k} }\,.
\end{equation}
In particular, if $X_k$ approaches zero as $k$ approaches infinity, then we have
\begin{equation}
\sum\limits_{n = 1}^\infty {\left( {X_n  - X_{n + m} } \right)}  = \sum\limits_{n = 1}^m {X_n}\,,
\end{equation}
\begin{equation}
\sum\limits_{n = 1}^\infty {( - 1)^{n - 1} \left( {X_n  - X_{n + m} } \right)}  = \sum\limits_{n = 1}^m {( - 1)^{n - 1} X_n },\quad\text{if $m$ is even}
\end{equation}
and
\begin{equation}
\sum\limits_{n = 1}^\infty {( - 1)^{n - 1} \left( {X_n  + X_{n + m} } \right)}  = \sum\limits_{n = 1}^m {( - 1)^{n - 1} X_n },\quad\text{if $m$ is odd}\,.
\end{equation}
Lemma \ref{lemma.id}, identities \eqref{eq.gdtp5sa}, \eqref{eq.jf9ybp9}, \eqref{eq.ec853w8}, \eqref{eq.f94zdr0} and the telescoping summation identity \eqref{eq.s1qn9qb} lead to:
\begin{theorem}
The following identities hold for non-negative integers $t$ and $m$:
\begin{equation}\label{eq.u408m58}
\sum\limits_{n = 1}^t {\tan ^{ - 1} \frac{{F_{2m} }}{{F_{2n + 2m - 1} }}}  = \sum\limits_{n = 1}^m {\tan ^{ - 1} \frac{{L_m }}{{L_{2n + m - 1} }}}  - \sum\limits_{n = 1}^m {\tan ^{ - 1} \frac{{L_m }}{{L_{2n + 2t + m - 1} }}},\quad\text{$m$ even}\,,
\end{equation}

\begin{equation}\label{eq.i2yrfn1}
\sum\limits_{n = 1}^t {\tan ^{ - 1} \frac{{F_{2m} }}{{F_{2n + 2m - 1} }}} = \sum\limits_{n = 1}^m {\tan ^{ - 1} \frac{{F_m }}{{F_{2n + m - 1} }}} - \sum\limits_{n = 1}^m {\tan ^{ - 1} \frac{{F_m }}{{F_{2n + 2t + m - 1} }}},\quad\text{$m$ odd}\,,
\end{equation}

\begin{equation}
\sum\limits_{n = 1}^t {\tan ^{ - 1} \frac{{L_m^2 L_{2n + 2m} }}{{5F_{2n + 2m}^2 }}} = \sum\limits_{n = 1}^m {\tan ^{ - 1} \frac{{L_m }}{{L_{2n + m} }}} - \sum\limits_{n = 1}^m {\tan ^{ - 1} \frac{{L_m }}{{L_{2n + 2t + m} }}}\,,\quad\text{$m$ odd}\,,
\end{equation}

\begin{equation}
\sum\limits_{n = 1}^t {\tan ^{ - 1} \frac{{F_m^2 L_{2n + 2m} }}{{F_{2n + 2m}^2 }}} = \sum\limits_{n = 1}^m {\tan ^{ - 1} \frac{{F_m }}{{F_{2n + m} }}} - \sum\limits_{n = 1}^m {\tan ^{ - 1} \frac{{F_m }}{{F_{2n + 2t + m} }}}\,,\quad\text{$m$ even}\,.
\end{equation}
\end{theorem}

In the limit as $t\to\infty$, we have:
\begin{corollary}
The following identities hold for positive integers $m$:
\begin{equation}\label{eq.g5vsa8d}
\sum\limits_{n = 1}^\infty {\tan ^{ - 1} \frac{{F_{2m} }}{{F_{2n + 2m - 1} }}}  = \sum\limits_{n = 1}^m {\tan ^{ - 1} \frac{{L_m }}{{L_{2n + m - 1} }}},\quad\text{$m$ even}\,,
\end{equation}

\begin{equation}\label{eq.p48nihu}
\sum\limits_{n = 1}^\infty {\tan ^{ - 1} \frac{{F_{2m} }}{{F_{2n + 2m - 1} }}} = \sum\limits_{n = 1}^m {\tan ^{ - 1} \frac{{F_m }}{{F_{2n + m - 1} }}},\quad\text{$m$ odd}\,,
\end{equation}

\begin{equation}
\sum\limits_{n = 1}^\infty {\tan ^{ - 1} \frac{{L_m^2 L_{2n + 2m} }}{{5F_{2n + 2m}^2 }}} = \sum\limits_{n = 1}^m {\tan ^{ - 1} \frac{{L_m }}{{L_{2n + m} }}}\,,\quad\text{$m$ odd}\,,
\end{equation}

\begin{equation}
\sum\limits_{n = 1}^\infty {\tan ^{ - 1} \frac{{F_m^2 L_{2n + 2m} }}{{F_{2n + 2m}^2 }}} = \sum\limits_{n = 1}^m {\tan ^{ - 1} \frac{{F_m }}{{F_{2n + m} }}}\,,\quad\text{$m$ even}\,.
\end{equation}

\end{corollary}
Lemma \ref{lemma.id} and identities \eqref{eq.gdtp5sa}, \eqref{eq.ryx4qng}, \eqref{eq.f94zdr0}, \eqref{eq.i97s6g6} and the telescoping summation identities \eqref{eq.n5lm399} and \eqref{eq.e3jcedw} give:
\begin{theorem}\label{theorem.q30p5d6}
The following identities hold for non-negative integers $t$ and $m$:
\begin{equation}\label{eq.tn4uryc}
\begin{split}
\sum\limits_{n = 1}^t {(-1)^{n-1}\tan ^{ - 1} \frac{{F_{2m} }}{{F_{2n + 2m - 1} }}}  &= \sum\limits_{n = 1}^m {(-1)^{n-1}\tan ^{ - 1} \frac{{L_m }}{{L_{2n + m - 1} }}}\\
&\qquad  + (-1)^{t-1}\sum\limits_{n = 1}^m {(-1)^{n-1}\tan ^{ - 1} \frac{{L_m }}{{L_{2n + 2t + m - 1} }}}\,,
\end{split}
\end{equation}
\begin{equation}
\begin{split}
\sum\limits_{n = 1}^t {(-1)^{n -1 } \tan ^{ - 1} \frac{{F_m^2 L_{2n + 2m} }}{{F_{2n + 2m}^2 }}} &= \sum\limits_{n = 1}^m {(-1)^{n -1 } \tan ^{ - 1} \frac{{F_m }}{{F_{2n + m} }}}\\
&\qquad + (-1)^{t-1} \sum\limits_{n = 1}^m {(-1)^{n -1 } \tan ^{ - 1} \frac{{F_m }}{{F_{2n + 2t + m} }}}\,.
\end{split}
\end{equation}
\end{theorem}
In the limit as $t\to\infty$ in the identities of Theorem \ref{theorem.q30p5d6}, we have:
\begin{corollary}
The following identities hold for non-negative integers $m$:
\begin{equation}\label{eq.jpvgd6p}
\sum\limits_{n = 1}^\infty {(-1)^{n-1}\tan ^{ - 1} \frac{{F_{2m} }}{{F_{2n + 2m - 1} }}}  = \sum\limits_{n = 1}^m {(-1)^{n-1}\tan ^{ - 1} \frac{{L_m }}{{L_{2n + m - 1} }}},
\end{equation}
\begin{equation}
\sum\limits_{n = 1}^\infty {(-1)^{n -1 } \tan ^{ - 1} \frac{{F_m^2 L_{2n + 2m} }}{{F_{2n + 2m}^2 }}} = \sum\limits_{n = 1}^m {(-1)^{n -1 } \tan ^{ - 1} \frac{{F_m }}{{F_{2n + m} }}}\,.
\end{equation}
\end{corollary}
\section{Generalizations of the identities of Hoggatt and Ruggles}
\begin{lemma}
Let $(X_i)$ be a sequence of real numbers and let $m$ and $t$ be integers. Then,
\begin{equation}\label{eq.y808d4b}
2\sum\limits_{n = 1}^t {X_n }  = \sum\limits_{n = 1}^t {(X_n + X_{n + m} )}  + \sum\limits_{n = 1}^m {X_n }  - \sum\limits_{n = t + 1}^{t + m} {X_n }\,,
\end{equation}
\begin{equation}\label{eq.ckksnfx}
2\sum\limits_{n = 1}^t {( - 1)^{n - 1} X_n }  = \sum\limits_{n = 1}^t {( - 1)^{n - 1} (X_n  - X_{n + m} )}  + \sum\limits_{n = 1}^m {( - 1)^{n - 1} X_n }  - \sum\limits_{n = t + 1}^{t + m} {( - 1)^{n - 1} X_n },\text{ $m$ odd}\,,
\end{equation}
\begin{equation}\label{eq.jg71xlc}
2\sum\limits_{n = 1}^t {( - 1)^{n - 1} X_n }  = \sum\limits_{n = 1}^t {( - 1)^{n - 1} (X_n  + X_{n + m} )}  + \sum\limits_{n = 1}^m {( - 1)^{n - 1} X_n }  - \sum\limits_{n = t + 1}^{t + m} {( - 1)^{n - 1} X_n },\text{ $m$ even}\,.
\end{equation}
\end{lemma}
\begin{proof}
We have
\[
\begin{split}
\sum\limits_{n = 1}^t {(X_{n + m}  + X_n )}  &= \sum\limits_{n = m + 1}^{t + m} {X_n }  + \sum\limits_{n = 1}^t {X_n }\\ 
 &= \sum\limits_{n = 1}^{t + m} {X_n }  - \sum\limits_{n = 1}^m {X_n }  + \sum\limits_{n = 1}^t {X_n }\\ 
&= 2\sum\limits_{n = 1}^t {X_n }  + \sum\limits_{n = t + 1}^{t + m} {X_n }  - \sum\limits_{n = 1}^m {X_n }\,, 
\end{split}
\]
from which identity \eqref{eq.y808d4b} follows. Identities \eqref{eq.ckksnfx} and \eqref{eq.jg71xlc} are proved when we replace $X_n$ by $(-1)^nX_n$ in identity \eqref{eq.y808d4b}.
\end{proof}
\begin{theorem}\label{theorem.cf5bgud}
The following identities hold for $t$ and $m$ non-negative integers:
\begin{equation}\label{eq.fmpz6qn}
\begin{split}
2\sum\limits_{n = 1}^t {\tan ^{ - 1} \frac{{L_m }}{{L_{2n + m - 1} }}}  &= \sum\limits_{n = 1}^m {\tan ^{ - 1} \frac{2}{{L_{2n - 1} }}}  - \sum\limits_{n = 1}^m {\tan ^{ - 1} \frac{{F_m }}{{F_{2n + 2t + m - 1} }}}\\
&\qquad- \sum\limits_{n = t + 1}^{t + m} {\tan ^{ - 1} \frac{{L_m }}{{L_{2n + m - 1} }}},\quad\text{ $m$ odd}\,,
\end{split}
\end{equation}

\begin{equation}\label{eq.rnshzqs}
\begin{split}
2\sum\limits_{n = 1}^t {\tan ^{ - 1} \frac{{F_m }}{{F_{2n + m - 1} }}}  &= \sum\limits_{n = 1}^m {\tan ^{ - 1} \frac{2}{{L_{2n - 1} }}}  - \sum\limits_{n = 1}^m {\tan ^{ - 1} \frac{{L_m }}{{L_{2n + 2t + m - 1} }}}\\
&\qquad- \sum\limits_{n = t + 1}^{t + m} {\tan ^{ - 1} \frac{{F_m }}{{F_{2n + m - 1} }}},\quad\text{ $m$ even}\,,
\end{split}
\end{equation}

\begin{equation}\label{eq.evkxy48}
\begin{split}
2\sum\limits_{n = 1}^t {( - 1)^{n - 1} \tan ^{ - 1} \frac{{F_m }}{{F_{2n + m - 1} }}}  &= \sum\limits_{n = 1}^m {( - 1)^{n - 1} \tan ^{ - 1} \frac{2}{{L_{2n - 1} }}}\\
&\quad  + ( - 1)^{t - 1} \sum\limits_{n = 1}^m {( - 1)^{n - 1} \tan ^{ - 1} \frac{{L_m }}{{L_{2n + 2t + m - 1} }}}\\
&\qquad- \sum\limits_{n = t + 1}^{t + m} {( - 1)^{n - 1} \tan ^{ - 1} \frac{{F_m }}{{F_{2n + m - 1} }}}\,.
\end{split}
\end{equation}
\end{theorem}
\begin{proof}
Setting $X_n=\tan^{-1}({L_m/L_{2n + m - 1}})$ in \eqref{eq.y808d4b} and making use of \eqref{eq.ryx4qng} gives
\[
\begin{split}
2\sum\limits_{n = 1}^t {\tan ^{ - 1} \frac{{L_m }}{{L_{2n + m - 1} }}}  &= \sum\limits_{n = 1}^t {\tan ^{ - 1} \frac{{F_{2m} }}{{F_{2n + 2m - 1} }}}  + \sum\limits_{n = 1}^m {\tan ^{ - 1} \frac{{L_m }}{{L_{2n + m - 1} }}}\\
&\qquad- \sum\limits_{n = t + 1}^{t + m} {\tan ^{ - 1} \frac{{L_m }}{{L_{2n + m - 1} }}}\,,
\end{split}
\]
which by \eqref{eq.i2yrfn1} is
\[
\begin{split}
2\sum\limits_{n = 1}^t {\tan ^{ - 1} \frac{{L_m }}{{L_{2n + m - 1} }}}  &= \sum\limits_{n = 1}^m {\tan ^{ - 1} \frac{{F_m }}{{F_{2n + m - 1} }}}  - \sum\limits_{n = 1}^m {\tan ^{ - 1} \frac{{F_m }}{{F_{2n + 2t + m - 1} }}}\\
&\qquad+ \sum\limits_{n = 1}^m {\tan ^{ - 1} \frac{{L_m }}{{L_{2n + m - 1} }}}  - \sum\limits_{n = t + 1}^{t + m} {\tan ^{ - 1} \frac{{L_m }}{{L_{2n + m - 1} }}}\,, 
\end{split}
\]
from which identity \eqref{eq.fmpz6qn} follows when we make use of idetity \eqref{eq.gxld0li}.

\medskip

The proof of \eqref{eq.rnshzqs} is similar. Use $X_n=\tan^{-1}({F_m/F_{2n + m - 1}})$ in \eqref{eq.y808d4b} and make use of \eqref{eq.imhso0b}, \eqref{eq.u408m58} and \eqref{eq.gxld0li}.

\medskip

To prove \eqref{eq.evkxy48}, first set $X_n=\tan^{-1}({F_m/F_{2n + m - 1}})$ in \eqref{eq.ckksnfx} and use \eqref{eq.jf9ybp9}, \eqref{eq.tn4uryc} and \eqref{eq.gxld0li}. This gives
\begin{equation}\label{eq.bnei8i0}
\begin{split}
2\sum\limits_{n = 1}^t {( - 1)^{n - 1} \tan ^{ - 1} \frac{{F_m }}{{F_{2n + m - 1} }}}  &= \sum\limits_{n = 1}^m {( - 1)^{n - 1} \tan ^{ - 1} \frac{2}{{L_{2n - 1} }}}\\
&\quad  + ( - 1)^{t - 1} \sum\limits_{n = 1}^m {( - 1)^{n - 1} \tan ^{ - 1} \frac{{L_m }}{{L_{2n + 2t + m - 1} }}}\\
&\qquad- \sum\limits_{n = t + 1}^{t + m} {( - 1)^{n - 1} \tan ^{ - 1} \frac{{F_m }}{{F_{2n + m - 1} }}},\quad\text{ $m$ odd}\,.
\end{split}
\end{equation}
Next set $X_n=\tan^{-1}({F_m/F_{2n + m - 1}})$ in \eqref{eq.jg71xlc} and use \eqref{eq.imhso0b}, \eqref{eq.tn4uryc} and \eqref{eq.gxld0li}. This gives
\begin{equation}\label{eq.ccvkpe2}
\begin{split}
2\sum\limits_{n = 1}^t {( - 1)^{n - 1} \tan ^{ - 1} \frac{{F_m }}{{F_{2n + m - 1} }}}  &= \sum\limits_{n = 1}^m {( - 1)^{n - 1} \tan ^{ - 1} \frac{2}{{L_{2n - 1} }}}\\
&\quad  + ( - 1)^{t - 1} \sum\limits_{n = 1}^m {( - 1)^{n - 1} \tan ^{ - 1} \frac{{L_m }}{{L_{2n + 2t + m - 1} }}}\\
&\qquad- \sum\limits_{n = t + 1}^{t + m} {( - 1)^{n - 1} \tan ^{ - 1} \frac{{F_m }}{{F_{2n + m - 1} }}},\quad\text{ $m$ even}\,.
\end{split}
\end{equation}
Identities \eqref{eq.bnei8i0} and \eqref{eq.ccvkpe2} imply identity \eqref{eq.evkxy48}.
\end{proof}
\begin{lemma}\label{lemma.af897g0}
Let $X_i$ be a sequence of real numbers such that $\lim\limits_{i\to\infty}{X_i}=0$. Then,
\begin{equation}\label{eq.ulwifpc}
2\sum\limits_{n = 1}^\infty  {X_n }  = \sum\limits_{n = 1}^m {X_n } + \sum\limits_{n = 1}^\infty  {(X_n + X_{n + m} )},\text{ $m$ integer}\,, 
\end{equation}
\begin{equation}\label{eq.yxh891p}
2\sum\limits_{n = 1}^\infty  {(-1)^{n-1}X_n }  = \sum\limits_{n = 1}^m {(-1)^{n-1}X_n } + \sum\limits_{n = 1}^\infty  {(-1)^{n-1}(X_n - X_{n + m})},\text{ $m$ odd}\,, 
\end{equation}
\begin{equation}\label{eq.hrweot0}
2\sum\limits_{n = 1}^\infty  {(-1)^{n-1}X_n }  = \sum\limits_{n = 1}^m {(-1)^{n-1}X_n } + \sum\limits_{n = 1}^\infty  {(-1)^{n-1}(X_n + X_{n + m})},\text{ $m$ even}\,. 
\end{equation}
\end{lemma}
Taking limit as $t$ goes to infinity in Theorem \ref{theorem.cf5bgud} or using Lemma \ref{lemma.af897g0} with the $X_n$ choices in the proof of Theorem \ref{theorem.cf5bgud}, we have:
\begin{corollary}
The following identities hold for non-negative integers $m$:
\begin{equation}
2\sum\limits_{n = 1}^\infty  {\tan ^{ - 1} \frac{{L_m }}{{L_{2n + m - 1} }}}  = \sum\limits_{n = 1}^m {\tan ^{ - 1} \frac{2}{{L_{2n - 1} }}},\text{ $m$ odd}\,,
\end{equation}
\begin{equation}\label{eq.yun9igl}
2\sum\limits_{n = 1}^\infty  {\tan ^{ - 1} \frac{{F_m }}{{F_{2n + m - 1} }}}  = \sum\limits_{n = 1}^m {\tan ^{ - 1} \frac{2}{{L_{2n - 1} }}},\text{ $m$ even}\,,
\end{equation}
\begin{equation}\label{eq.otv9p5v}
2\sum\limits_{n = 1}^\infty  {(-1)^{n-1}\tan ^{ - 1} \frac{{F_m }}{{F_{2n + m - 1} }}}  = \sum\limits_{n = 1}^m {(-1)^{n-1}\tan ^{ - 1} \frac{2}{{L_{2n - 1} }}}\,.
\end{equation}
\end{corollary}

\hrule

\medskip

\noindent 2010 {\it Mathematics Subject Classification}:
Primary 11B39; Secondary 11B37.

\noindent \emph{Keywords: }
Fibonacci number, Lucas number, Lucas sequence, summation identity, partial sum

\hrule

\medskip

\noindent Concerned with sequences: 
A000032, A000045


\end{document}